\pdfoutput=1
\documentclass[11pt,letterpaper]{article}

\usepackage[utf8]{inputenc}

\usepackage[left=1in,right=1in,top=1in,bottom=1in]{geometry}

\usepackage{bbm}

\usepackage{amssymb}
\usepackage{amsmath}
\usepackage{amsthm}
\allowdisplaybreaks

\usepackage{stmaryrd}

\usepackage{tikz}
\usepackage{tikz-cd}

\usepackage{enumitem}

\usepackage{mathtools}

\usepackage{graphicx}

\usepackage{pdflscape}

\usepackage{verbatim}

\usepackage[style=alphabetic,url=false]{biblatex}
\usepackage{hyperref}

\usepackage[english,noabbrev,sort]{cleveref}

\let\C\undefined
\usepackage{my-math-definitions} 

\newtheorem{theorem}{Theorem}
\newtheorem{lemma}[theorem]{Lemma}
\newtheorem{example}[theorem]{Example}
\newtheorem{remark}[theorem]{Remark}
\newtheorem{corollary}[theorem]{Corollary}
\newtheorem{definition}[theorem]{Definition}

\newtheorem{conjecture}[theorem]{Conjecture}
\newtheorem{mistake}[theorem]{Heuristic}
\numberwithin{theorem}{section}
\numberwithin{equation}{section}

\author{Fabian Gundlach}
\title{Malle's conjecture with multiple invariants}

\newcommand{\sep}{\mathrm{sep}}

\DeclareMathOperator{\triv}{triv}

\DeclareMathOperator{\inv}{inv}
\DeclareMathOperator{\Frob}{Frob}

\newcommand{\fartin}{f_{\textnormal{Artin}}}
\newcommand{\ffine}{f_{\textnormal{fine}}}
\newcommand{\condartin}{\cond_{\textnormal{Artin}}}
\newcommand{\condfine}{\cond_{\textnormal{fine}}}

\begin{document}

\maketitle

\begin{abstract}
We define invariants $\inv_1,\dots,\inv_m$ of Galois extensions of number fields with a fixed Galois group. Then, we propose a heuristic in the spirit of Malle's conjecture which asymptotically predicts the number of extensions that satisfy $\inv_i\leq X_i$ for all $X_i$. The resulting conjecture is proved for abelian Galois groups. We also describe refined Artin conductors that carry essentially the same information as the invariants $\inv_1,\dots,\inv_m$.
\end{abstract}

\section{Introduction}

Let $G$ be a finite group and let $K$ be a number field.

In \cite{malle-distribution-of-galois-groups-2}, Malle made the following conjecture regarding the number of field $G$-extensions of $K$ with a bounded invariant:

\begin{conjecture}\label{malle-conjecture}
Let $H$ be a subgroup of $G$ with $\bigcap_{g\in G}gHg^{-1} = 1$. Let $N(X)$ be the number of Galois extensions $L$ of $K$ with Galois group $G$ and $|\Nm(\disc(L^H|K))|\leq X$. There are explicit constants $a\geq1$, $b\geq1$ and a constant $C>0$ such that
\[
N(X) \sim C X^{1/a} (\log X)^{b-1}
\]
for $X\ra\infty$.
\end{conjecture}

Malle's conjecture has been proven in certain special cases, but remains open in general. For example, Wright proved the conjecture for abelian groups in \cite{wright-counting-abelian-extensions}. See \cite[page 2]{koymans-pagano-nilpotent-1} for a list of known cases.

In \cite{heights-on-stacks}, Ellenberg, Satriano, and Zureick-Brown propose a variant of Manin's conjecture for stacks and interpret Malle's conjecture as a special case. In \cite{manin-on-stacks}, Darda and Yasuda give a more precise version of Manin's conjecture for stacks.

Malle heuristically justified his conjecture. The same heuristic has been applied to invariants that are not of the form $|\Nm(\disc(L^H|K))|$, for example Artin conductors (see \cite{counting-d4-by-conductor} and also \cite{johnson-thesis}) or the product of the norms of all ramified primes (see \cite{wood-probabilities-of-local-behaviors}).

Although Malle's conjecture is about field extensions of $K$ with Galois group $G$, we will in this paper count étale $G$-extensions of $K$ with a prescribed action of $G$. We will count these (étale) $G$-extensions with a weight proportional to the inverse of the number of automorphisms (as a $G$-extension). Every étale $G$-extension $L$ is induced from a field $H$-extension for some subgroup $H$ of $G$, and the automorphism group of $L$ as a $G$-extension is then isomorphic to the centralizer of $H$ in $G$. The inclusion--exclusion principle can be used to translate between the problem of counting field $G$-extensions and the problem of counting étale $G$-extensions. We work with étale extensions because it seems more natural\footnote{although as we explain in \cref{counterexample-section} less correct} to apply the heuristics for Malle's conjecture (particularly as explained in \cite{bhargava-mass-formulae} and \cite[section 10]{wood-aws-lecture-notes}) in that setting.

In \cref{invariants-section}, we define a system of \emph{basic invariants} $\inv_1(L),\dots,\inv_m(L)\in\Z_{\geq1}$ of (étale) $G$-extensions $L$ of~$K$. Ignoring the contribution of wildly ramified primes, the invariants listed above (Artin conductors and the product of the norms of all ramified primes) can all be expressed in the form $\inv_1^{h_1}\cdots\inv_m^{h_m}$ for constants $h_1,\dots,h_m>0$.

If $K=\Q$, we have one basic invariant $\inv_i$ for each conjugacy class of nontrivial cyclic subgroups $I$ of $G$: the product of all tamely ramified prime numbers $p$ whose inertia groups for $L|K$ lie in this conjugacy class.
These invariants were used traditionally, for example in \cite[page 163]{ellenberg-venkatesh-counting-galois-extensions} and \cite{wood-probabilities-of-local-behaviors}.

For some larger base fields $K$, however, it makes sense to go beyond just the inertia group when describing the way a prime ramifies, and we construct a larger number of basic invariants $\inv_1,\dots,\inv_m$ in \cref{invariants-section}. For example, if $\#G=n$ and $K$ contains a primitive $n$-th root of unity, then the inertia groups have a canonical generator. We then obtain one basic invariant for each conjugacy class of elements $g$ of $G$.

The basic invariants are in noncanonical bijection with the nontrivial irreducible representations of $G$ defined over $K$. In \cref{fine-conductor-section}, we associate to any representation $\rho$ of $G$ defined over $K$ an invariant of $G$-extensions of $K$, which we call the \emph{fine Artin conductor}. These conductors in general capture more information than Artin conductors (which can be thought of as associated to representations defined over $\Q$). They in general also capture more information than the weights defined in \cite{wood-yasuda-1}, which are used as the height functions in the Manin-on-stacks paper \cite{heights-on-stacks}. The fine Artin conductor agrees with the Artin conductor when $\rho$ is defined over $\Q$. (See \Cref{fine_same}.) They share many properties with Artin conductors. (See \Cref{fine_properties}.) In \Cref{linear-independence}, we show that the fine Artin conductors are essentially equivalent to the basic invariants $\inv_1,\dots,\inv_m$.

Rather than fixing a particular invariant $\inv$ and studying the distribution of the numbers $\inv(L)$ for $G$-extensions $L|K$ (as in Malle's conjecture), Ellenberg and Venkatesh in \cite{ellenberg-venkatesh-counting-galois-extensions} suggested to look at the distribution of the points $(\inv_1(L),\dots,\inv_m(L))\in\Z_{\geq1}^m$. How many $G$-extensions $L$ of $K$ are there such that the corresponding point lies in a certain (large) region?

In \cref{heuristic-section}, we explain how a heuristic in the spirit of \cite{malle-distribution-of-galois-groups-2}, \cite{bhargava-mass-formulae}, and \cite[section 10]{wood-aws-lecture-notes} suggests the following:

\begin{mistake}\label{wrong-conjecture}
There is a constant $C>0$ such that if $X_1,\dots,X_m$ all go to~$\infty$, then
\[
\sum_{\substack{\textnormal{$G$-ext. $L|K$}:\\\inv_i(L)\leq X_i\forall i}} \frac1{\#\Aut(L)} \sim CX_1\cdots X_m.
\]
\end{mistake}

Over the base field $K=\Q$, and discounting non-field extensions, \Cref{wrong-conjecture} was stated as Question 4.5 in \cite{ellenberg-venkatesh-counting-galois-extensions}. (They only considered the invariants associated to conjugacy classes of nontrivial cyclic subgroups of~$G$. For other base fields $K$, using this smaller set of invariants, \Cref{wrong-conjecture} would in general be wrong even for abelian groups $G$.)

Rather than letting the upper bounds $X_i$ on all invariants $\inv_i(L)$ go to~$\infty$, we could let only some of the upper bounds go to $\infty$ and fix the remaining invariants. One would expect the following:

\begin{mistake}\label{wrong-boundary-conjecture}
Let $T$ be a subset of $\{1,\dots,m\}$. For each $i\in T$, fix a number $x_i\geq1$. Then, there is a constant $C\geq0$ such that if the numbers $X_i$ for $i\notin T$ all go to $\infty$, then
\[
\sum_{\substack{\textnormal{$G$-ext. $L|K$}:\\\inv_i(L)=x_i\forall i\in T\\\inv_i(L)\leq X_i\forall i\notin T}} \frac1{\#\Aut(L)} \sim C\prod_{i\notin T}X_i.
\]
\end{mistake}

In \cite[Theorems 8, 13, and 18]{shankar-thorne-cubic-fields}, Shankar and Thorne recently proved \Cref{wrong-boundary-conjecture} for $G=S_3$ and $K=\Q$. In \cref{abelian-section}, we prove:
\begin{theorem}
\Cref{wrong-conjecture} and \Cref{wrong-boundary-conjecture} hold for all finite abelian groups $G$ and number fields $K$.
\end{theorem}
The author is not aware of any other cases in which even \Cref{wrong-conjecture} has been proven.

In \cite{klueners-counterexample-malles-conjecture}, Klüners presented a counterexample to Malle's conjecture. It is not hard to see that it also disproves \Cref{wrong-conjecture}! However, as we explain in \cref{counterexample-section}, all counterexamples known to the author have some bounded invariant, so the following conjecture may still hold.

\begin{conjecture}\label{bold-conjecture}
Fix some number $0<\delta<1$. There is a constant $C>0$ such that if $X_1,\dots,X_m$ all go to $\infty$, then
\[
\#\{L\mid \delta X_i<\inv_i(L)\leq X_i\textnormal{ for all }i=1,\dots,m\} \sim C X_1\cdots X_m.
\]
\end{conjecture}

As we explain in \Cref{non-field-remark}, the étale $G$-extensions $L$ in \Cref{bold-conjecture} are in fact fields if $\delta X_1,\dots,\delta X_m>1$. We omitted the weights $\frac1{\#\Aut(L)}$ in \Cref{bold-conjecture} because all field $G$-extensions have the same number of automorphisms. (The automorphism group is isomorphic to the center of~$G$.)

For any group $G$, \Cref{wrong-conjecture} implies \Cref{bold-conjecture} by the inclusion--exclusion principle. In \cref{relationship-with-malle-section}, we comment on the relationship between \Cref{wrong-conjecture} or \Cref{wrong-boundary-conjecture} and counting $G$-extensions which satisfy an inequality $\inv_1^{h_1}\cdots\inv_m^{h_m}\leq X$.

In \cite[section 4]{peyre-beyond-heights}, Peyre discusses a variant of Manin's conjecture in which restrictions are placed on all heights associated to a basis of the Picard group, analogously to \Cref{bold-conjecture}. As Peyre points out in \cite[section 4.4]{peyre-beyond-heights}, there are counterexamples to this variant of Manin's conjecture. For example, if $V$ is a smooth cubic volume, the Picard rank of $V$ is $1$ (corresponding to just one invariant, $m=1$), and for the height function $H(x)$ associated to any ample line bundle, the set of points $\{v\in V(\Z)\}$ of height $\delta X<H(v)\leq X$ grows more quickly as $X\ra\infty$ than predicted by Manin's conjecture.

\paragraph{Acknowledgments} The author is grateful to Brandon Alberts, Jordan Ellenberg, Emmanuel Peyre, Frank Thorne, Melanie Matchett Wood, and Takehiko Yasuda for helpful remarks. Most of the discussions took place at the 2022 Simons Symposium on Geometry of Arithmetic Statistics, so the author would also like to extend his gratitude to the organizers and to the Simons Foundation.

\paragraph{Notation} We denote the residue field of a prime $\p$ by $k_\p$ and a uniformizer by $\pi_\p$. For a $G$-extension $L|K$ and corresponding extension of primes $\q|\p$, we denote the inertia group by $I(\q|\p)\subseteq G$ and the inertia degree by $e(\q|\p)$. We denote the completion of a number field $K$ at a place $v$ by $K_v$. For any $e\geq1$, we denote by $\mu_e$ the group of $e$-th roots of unity in the algebraic closure $\ol K$.

\section{Invariants}\label{invariants-section}

In this section, we describe the basic invariants $\inv_1,\dots,\inv_m$.

The idea is as follows: If $\q|\p$ is a tamely ramified extension of primes in a $G$-extension $L|K$, then the cyclic group $I(\q|\p)$ comes with a natural injective group homomorphism $\beta_{\q|\p}:I(\q|\p)\ra k_\q^\times$ given by $g\mapsto g(\pi_\q)/\pi_\q\bmod\q$. (This homomorphism is independent of the choice of uniformizer $\pi_\q$.) Let $e=e(\q|\p)$ be the inertia degree. The map $\beta_{\q|\p}$ is then an isomorphism onto the group $\mu_{\q,e}$ of $e$-th roots of unity in $k_\q^\times$.

If $\mu_e\subseteq K$, then the $e$-th roots of unity in $K$ can be naturally identified with the $e$-th roots of unity modulo any prime $\q$; we have a natural isomorphism $\mu_{\q,e}\ra\mu_e$. Hence, $\beta_{\q|\p}$ can be regarded as an isomorphism $I(\q|\p)\ra\mu_e$. We then call the inertia group together with this isomorphism, up to conjugation by $G$, the ramification type of $\p$ in $L|K$.

In general, if $\mu_e\nsubseteq K$, the $e$-th roots of unity modulo different primes $\q$ can only be reasonably identified modulo the action of $\Gal(K(\mu_e)|K)$.

This leads to the following definitions.

\begin{definition}\label{ramification_types_def}
Consider the set of pairs $(I,\gamma)$ where $I$ is a cyclic subgroup of $G$ of size $e$ and $\gamma$ is a group isomorphism $I\ra\mu_e$. The group $G$ acts on this set by conjugation: $g.(I,\gamma)=(gIg^{-1},\gamma')$ with $\gamma'(ghg^{-1})=\gamma(h)$. The group $U_e:=\Gal(K(\mu_e)|K)$ acts by $\sigma.(I,\gamma)=(I,\sigma\circ\gamma)$. Note that the two actions commute, so they combine to an action of $G\times U_e$. We call an orbit $[I,\gamma]$ of pairs $(I,\gamma)$ under this action of $G\times U_e$ a \emph{ramification type}.
\end{definition}

\begin{definition}\label{ramification_type_def}
Let $L$ be a $G$-extension of $K$ with a tamely ramified extension $\q|\p$ of primes of inertia degree $e$ and with inertia group $I$. (Hence, $e$ is not divisible by the residue field characteristic.) Let $\P$ be a prime of $L(\mu_e)$ dividing $\q$. For any $g\in I$, we obtain a unique root of unity $\gamma(g)\in\mu_e$ with $\gamma(g)\equiv\beta_{\q|\p}(g)\mod\P$. This defines a group isomorphism $\gamma:I\ra\mu_e$. The equivalence class $[I,\gamma]$ will be called the \emph{ramification type of $\p$ in $L$}.
\end{definition}

Note that the resulting ramification type $[I,\gamma]$ is independent of the choice of $\q$ and $\P$: The Galois group of $L(\mu_e)$ over $K$ is naturally a subgroup of $G\times U_e$. If an element $(g,\sigma)$ sends the prime $\P$ to the prime $\P'$ and $g$ sends $\q$ to $\q'$, then the resulting pair $(I',\gamma')$ for $\q',\P'$ is $(g,\sigma).(I,\gamma)$. Hence, the ramification type $[I,\gamma]$ is in fact independent of the choice of $\q$ and $\P$.

\begin{example}
If $K=\Q$, then $\Gal(K(\mu_e)|K)=(\Z/e\Z)^\times$ acts transitively on the primitive $e$-th roots of unity and therefore on the isomorphisms $\gamma:I\ra\mu_e$. Hence, the isomorphism $\gamma$ is redundant: Ramification types simply correspond to conjugacy classes of cyclic subgroups~$I$ of~$G$.

The ramification type of $\p$ in $L$ then corresponds to its inertia group (up to conjugation).
\end{example}

\begin{example}\label{enough-roots-of-unity-remark}
Let $n$ be the exponent of $G$. If $K$ contains all $n$-th roots of unity, then $\Gal(K(\mu_e)|K)$ is trivial and ramification types are in bijection with conjugacy classes of elements of $G$. (For each $e\mid n$, fix a primitive $e$-th root of unity $\zeta_e\in K$ and associate to $[I,\gamma]$ the conjugacy class containing $\gamma^{-1}(\zeta_e)$.)

The ramification type of $\p$ in $L$ then corresponds (up to conjugation) to the element $g$ of its inertia group such that $\beta_{\q|\p}(g)\equiv\zeta_e\mod\q$.
\end{example}

\begin{remark}\label{general-conj-remark}
Let $n$ be the exponent of $G$. For arbitrary $K$, the ramification types $[I,\gamma]$ are in (non-canonical) bijection with orbits of the action of $\Gal(K(\mu_n)|K)\subseteq(\Z/n\Z)^\times$ on the set of conjugacy classes of $G$ defined by $k.C=C^k$. (For each $e\mid n$, fix a $\Gal(K(\mu_e)|K)$-orbit $M_e$ of primitive $e$-th roots of unity and associate to $[I,\gamma]$ the set of conjugacy classes $C$ that contain an element $g$ of $\gamma^{-1}(M_e)$.) We point out here that the described orbits of conjugacy classes have already come up in the definition of the constant $b$ in Malle's conjecture. (See \cite[section 2]{malle-distribution-of-galois-groups-2}, where Malle refers to ``the inertia group generator'' lying in one of the $\Gal(K(\mu_e)|K)$-orbits of conjugacy classes.)
\end{remark}

\begin{example}
Let $n\geq1$ and $K=\Q(\zeta_n)$ and let $G=C_n$ be the cyclic group of order $n$. Fix a generator $\sigma$ of $G$. Let $x\in K^\times$ and consider the Kummer $G$-extension $L=K[T]/(T^n-x)$ of $K$ with the action of $G$ given by $\sigma(T)=\zeta_n\cdot T$. Let $\q|\p$ be an extension of primes of $L|K$, and assume that $\p$ doesn't divide $n$. Denote the $\p$-adic valuation of $x$ by $r$. Then, the inertia group $I=I(\q|\p)$ is generated by $\sigma^r$ and the map $\gamma:I\ra\mu_e$ is given by $\gamma(\sigma^r)=\zeta_n^{\gcd(r,n)}$.
\end{example}

Let us denote the ramification types with $I\neq1$ by $[I_1,\gamma_1],\dots,[I_m,\gamma_m]$. We then define the $i$-th basic invariant to be the product of the norms of all primes $\p$ of $K$ with ramification type $[I_i,\gamma_i]$:
\[
\inv_i(L) = \prod_{\p\textnormal{ of ram. type }[I_i,\gamma_i]} \Nm(\p)
\qquad\textnormal{for }i=1,\dots,m.
\]

\section{Known counterexamples}\label{counterexample-section}

In \cite{klueners-counterexample-malles-conjecture}, Klüners disproved Malle's conjecture but hypothesized (in section~3) that Malle's conjecture might be true if we omit from the count all Galois extensions $L$ of $K$ that contain a field $K\subsetneq K'\subseteq K(\zeta_n)$ with $n=\#G$. The following remark shows that these extensions $L$ are not counted in \Cref{bold-conjecture} if we let $X_1,\dots,X_m$ (and therefore $\inv_1(L),\dots,\inv_m(L)$) go to~$\infty$:
\begin{remark}
Fix a normal subgroup $H\subseteq G$ and a $G/H$-extension $K'$ of~$K$. Consider a ramification type $[I_i,\gamma_i]$ with $I_i\nsubseteq H$. If $L$ is a $G$-extension of~$K$ with $L^H = K'$, then $\inv_i(L) \mid \Nm(\disc(K'|K))$.
\end{remark}
\begin{proof}
If $\p$ is a prime of $L$ with ramification type $[I_i,\gamma_i]$ in $L$, then $K'=L^H$ is ramified at $\p$. Hence, the relative discriminant of $K'$ over $K$ is divisible by $\p$.
\end{proof}

In \Cref{wrong-conjecture} and \Cref{wrong-boundary-conjecture}, we count all étale $G$-extensions rather than only field extensions. It is well-known that this introduces additional failures of Malle's conjecture:
\begin{example}
Let $G=S_4$ be the group of permutations of $\{1,2,3,4\}$ and let $H$ be the stabilizer of $1$. The number of field $G$-extensions $L$ of $\Q$ with $\disc(L^H)\leq X$ is $\sim C\cdot X$ for some constant $C>0$ as predicted by Malle's conjecture. (See \cite{bhargava-quartic-counting}.) However, the number of étale $G$-extensions grows more quickly: Let $A\cong C_2\times C_2$ be the subgroup of $G$ generated by the permutations $(1\ 2)$ and $(3\ 4)$. For any two quadratic number fields $L_1,L_2$, their product $L_1\times L_2$ can be viewed as an étale $A$-extension of~$\Q$. Let $L$ be the induced $G$-extension. Then, $\disc(L^H) = \disc(L_1\times L_2) = \disc(L_1)\cdot\disc(L_2)$. The number of quadratic extensions of $\Q$ with absolute discriminant at most $Y$ is $\sim C' Y$, so the number of pairs $(L_1,L_2)$ as above with $|\disc(L_1)\cdot\disc(L_2)|\leq X$ is $\sim C'' X\log X$. In particular, there are at least $\sim C''X\log X$ étale $G$-extensions $L$ with $|\disc(L)|\leq X$, larger than predicted by Malle's conjecture.
\end{example}
However, the following remark shows that non-field $G$-extensions are not counted in \Cref{bold-conjecture} if we let $X_1,\dots,X_m$ go to $\infty$:
\begin{remark}\label{non-field-remark}
Fix a subgroup $H\subsetneq G$. It is well-known that $G$ is not the union of conjugates of $H$. (There are at most $\#G/\#H$ subgroups conjugate to $H$, each of which has $\#H$ elements. They all contain the identity.) Hence, there is a ramification type $[I_i,\gamma_i]$ such that $I_i$ is not a subgroup of any conjugate of $H$. If $L$ is the $G$-extension of~$K$ induced by any $H$-extension $L'$ of~$K$, then $\inv_i(L) = 1$.
\end{remark}
\begin{proof}
The étale extension $L$ is a product of $[G:H]$ copies of $L'$. The action of $H\subseteq G$ on one of the copies agrees with the action of $H$ on the $H$-extension $L'$. Consider any prime $\p$ of $K$ and let $\q|\p$ be a prime of this copy of $L'$. Then, the inertia group $I(\q|\p)$ for the $G$-extension $L|K$ is the same as the inertia group $I(\q|\p)$ for the $H$-extension $L'|K$. In particular, it is a subgroup of $H$. Hence, if $\p$ has inertia type $[I_i,\gamma_i]$ in $L$, then $I_i$ must be conjugate to the subgroup $I(\q|\p)$ of $H$.
\end{proof}

\section{A Tauberian theorem}

Both the heuristic in \cref{heuristic-section} and the proof for abelian groups in \cref{abelian-section} ultimately rely on a Tauberian theorem for Dirichlet series $D(s_1,\dots,s_m)$ in several variables. The author has not found a suitable multivariate variant of the Wiener--Ikehara theorem in the literature. We will instead make use of the following special case, which we reduce to the Wiener--Ikehara theorem for single-variable Dirichlet series via a sieving argument.

\begin{lemma}\label{tauberian-lemma}
Let $K$ be a number field and $m\geq0$. For every prime $\p$ of~$K$, let $b_\p,c_{\p,1},\dots,c_{\p,m}$ be complex numbers. Assume that $b_\p=1$ for all but finitely many $\p$ and that the numbers $c_{\p,i}$ are bounded. Let
\[
D_\p(s_1,\dots,s_m) = b_\p + c_{\p,1} \Nm(\p)^{-s_1} + \cdots + c_{\p,m} \Nm(\p)^{-s_m}
\]
and formally expand
\[
\prod_\p D_\p(s_1,\dots,s_m) = \sum_{x_1,\dots,x_m\geq1} a_{x_1,\dots,x_m}x_1^{-s_1}\cdots x_m^{-s_m}.
\]
Let
\[
D_{\p,i}(s) = 1 + c_{\p,i}\Nm(\p)^{-s}
\]
and assume $E_0$ is large enough so that $b_\p=1$ and $D_{\p,1}(1)\cdots D_{\p,m}(1)\neq0$ for all $\p$ with $\Nm(\p)>E_0$. For each $i$, assume that the function
\[
(s-1)\prod_{\p:\Nm(\p)>E_0} D_{\p,i}(s)
\]
can be holomorphically continued to $\{\Re(s)\geq1\}$ with value $r_i$ at $s=1$. Furthermore, assume there are bounded real numbers $c_{\p,i}'\geq|c_{\p,i}|$ such that, putting
\[
D_{\p,i}'(s) = 1+c_{\p,i}'\Nm(\p)^{-s},
\]
the function
\[
(s-1)\prod_{\p:\Nm(\p)>E_0} D_{\p,i}'(s)
\]
can be holomorphically continued to $\{\Re(s)\geq1\}$.

Then, as $X_1,\dots,X_m$ all go to $\infty$,
\[
\sum_{x_1\leq X_1,\dots,x_m\leq X_m} a_{x_1,\dots,x_m}
= C X_1\cdots X_m + o(X_1\cdots X_m)
\]
with
\[
C = r_1\cdots r_m \cdot \prod_{\p:\Nm(\p)\leq E_0} D_\p(1,\dots,1) \prod_{\p:\Nm(\p)>E_0} \frac{D_\p(1,\dots,1)}{D_{\p,1}(1)\cdots D_{\p,m}(1)}.
\]
\end{lemma}

\begin{remark}
Note that if $\Nm(\p)>E_0$, then
\[
\frac{D_\p(1,\dots,1)}{D_{\p,1}(1)\cdots D_{\p,m}(1)} = 1 + \O(\Nm(\p)^{-2}),
\]
so the product in the definition of $C$ converges.
\end{remark}

\begin{remark}
The assumptions imply that the function
\[
(s_1-1)\cdots(s_m-1)\prod_\p D_\p(s_1,\dots,s_m)
\]
can be holomorphically extended to $\{\Re(s_1),\dots,\Re(s_m)\geq1\}$ with value $C$ at $s_1=\cdots=s_m=1$.
\end{remark}

\begin{proof}[Proof of the lemma]

We will prove the theorem using a sieve up to primes of norm at most $E\geq E_0$. 

First, note that by assumption the Dirichlet series $\prod_{\p:\Nm(\p)>E}D_{\p,i}'(s)$, and therefore $F_{E,i}(s):=\prod_{\p:\Nm(\p)>E}D_{\p,i}(s)$, are absolutely convergent for $\Re(s)>1$. Also, both have a meromorphic continuation to $\{\Re(s)\geq1\}$ with at most a simple pole at $s=1$. The latter Dirichlet series $F_{E,i}(s)$ has residue $r_i/\prod_{\p:E_0<\Nm(\p)\leq E}D_{\p,i}(1)$ at $s=1$.

By the Wiener--Ikehara theorem (see Exercises 3.3.3 and 3.3.4 in \cite{murty-problems-in-analytic-number-theory}), the $x^{-s}$-coefficients with $x\leq X$ in the Dirichlet series $F_{E,i}(s)$ therefore sum up to
\[
X\cdot \left(\frac{r_i}{\prod_{\p:E_0<\Nm(\p)\leq E}D_{\p,i}(1)} + o_E(1)\right)
\]
as $X$ goes to $\infty$. (The rate of convergence of $o_E(1)$ to $0$ may depend on $E$.)

Consider the Dirichlet series
\begin{align*}
D_E(s_1,\dots,s_m)
&= \prod_{\p:\Nm(\p)\leq E} D_\p(s_1,\dots,s_m) \cdot \prod_{\p:\Nm(\p)>E} D_{\p,1}(s_1)\cdots D_{\p,m}(s_m) \\
&= \prod_{\p:\Nm(\p)\leq E} D_\p(s_1,\dots,s_m) \cdot F_{E,1}(s_1)\cdots F_{E,m}(s_m).
\end{align*}
Expanding the first (finite) product and looking at its (finitely many) summands one at a time, we conclude that the $x_1^{-s_1}\cdots x_m^{-s_m}$-coefficients in $D_E(s_1,\dots,s_m)$ with $x_1\leq X_1,\dots,x_m\leq X_m$ sum up to
\[
X_1\cdots X_m\cdot \left(\frac{r_1\cdots r_m \cdot \prod_{\p:\Nm(\p)\leq E} D_\p(1,\dots,1)}{\prod_{\p:E_0<\Nm(\p)\leq E}D_{\p,1}(1)\cdots D_{\p,m}(1)} + o_E(1)\right).
\]
If $\Nm(\p)>E_0$, we have
\[
D_{\p,1}(s_1)\cdots D_{\p,m}(s_m) - D_\p(s_1,\dots,s_m)
= \sum_{\substack{M:\\\#M\geq2}} \prod_{i\in M} c_{\p,i}\Nm(\p)^{-s_i},
\]
where the sum ranges over the subsets $M$ of $\{1,\dots,m\}$ of size at least $2$. It follows that the Dirichlet series coefficients in the difference
\[
D_E(s_1,\dots,s_m) - \prod_\p D_\p(s_1,\dots,s_m)
\]
are bounded in absolute value by the corresponding coefficients in the product
\[
\prod_{\p:\Nm(\p)\leq E_0} D_\p'(s_1,\dots,s_m) \cdot
\prod_{\p:\Nm(\p)>E_0} D_{\p,1}'(s_1)\cdots D_{\p,m}'(s_m) \cdot 
\sum_{\p_0:\Nm(\p_0)>E}\sum_{\substack{M:\\\#M\geq2}} \prod_{i\in M} c_{\p_0,i}'\Nm(\p)^{-s_i},
\]
where we put
\[
D_\p'(s_1,\dots,s_m) = |b_\p| + c_{\p,1}'\Nm(\p)^{-s_1} + \cdots + c_{\p,m}'\Nm(\p)^{-s_m}.
\]
Using the Wiener--Ikehara theorem as before, we see that the sum of the $x_1^{-s_1}\cdots x_m^{-s_m}$-coefficients with $x_1\leq X_1,\dots,x_m\leq X_m$ in this error term is
\[
\O(X_1\cdots X_m\sum_{\p_0:\Nm(\p_0)>E}\Nm(\p_0)^{-2}).
\]
(The constant here does not depend on $E$.)

The claim follows by letting $E$ go to $\infty$ since $\sum_{\p_0:\Nm(\p_0)>E}\Nm(\p_0)^{-2}$ converges to $0$.
\end{proof}

\begin{corollary}\label{tauberian-boundary-lemma}
Under the same assumptions as in the previous lemma: Let $T$ be a subset of $\{1,\dots,m\}$. For each $i\in T$, fix a number $x_i\geq1$. Then, there is a constant $C'\geq0$ such that if the numbers $X_i$ for $i\notin T$ all go to $\infty$, then
\[
\sum_{\substack{(x_i)_{i\notin T}:\\x_i\leq X_i\forall i\notin T}} a_{x_1,\dots,x_m}
= C' \prod_{i\notin T}X_i + o(\prod_{i\notin T}X_i).
\]
If $C'=0$ but $r_1\cdots r_m\neq0$ and the coefficients $b_\p$ and $c_{\p,i}$ are nonnegative real numbers, then the left-hand side is $0$.
\end{corollary}
\begin{proof}
When expanding $\prod_\p D_\p(s_1,\dots,s_m)$, to obtain a summand of the form $x_1^{-s_1}\cdots x_m^{-s_m}$ with fixed $x_i$ for $i\in T$, one can only select a summand of the form $c_{\p,i}\Nm(\p)^{-s_i}$ in $D_\p(s_1,\dots,s_m)$ if $\p$ divides $x_i$.

Hence, let us first choose which summands to use from $D_\p(s_1,\dots,s_m)$ for each prime $\p$ dividing some $x_i$ with $i\in T$. There are only finitely many choices. Then, one can apply the lemma, omitting those factors $D_\p(s_1,\dots,s_m)$ and omitting all summands for $i\notin T$ from the other factors $D_\p(s_1,\dots,s_m)$.

If there is a choice for which the resulting product is nonzero and we have $r_1\cdots r_m\neq0$ and the coefficients $c_{\p,i}$ are nonnegative real numbers, the constant from the lemma is in fact positive. This shows $C'>0$.
\end{proof}

\section{Heuristic}\label{heuristic-section}

We invoke the well-known correspondence between $G$-extensions of $K$ and continuous homomorphisms $\Gal(\ol K|K)\ra G$. Any $G$-extension corresponds to exactly $\#G/\#\Aut(L)$ homomorphisms. Thus, we extend the notion of ramification types and associated invariants to such homomorphisms. We furthermore extend it to $G$-extensions of a completion $K_v$ (and similarly to homomorphisms $\Gal(\ol K_v|K_v)\ra G$) in the natural way.

To approach \Cref{wrong-conjecture}, \Cref{wrong-boundary-conjecture}, or \Cref{bold-conjecture}, consider the Dirichlet series
\begin{align*}
D(s_1,\dots,s_m)
&= \sum_{\textnormal{$G$-ext.{} }L|K} \frac{1}{\#\Aut(L)}\cdot \inv_1(L)^{-s_1}\cdots\inv_m(L)^{-s_m} \\
&= \frac1{\#G}\sum_{f:\Gal(\ol K|K)\ra G} \inv_1(f)^{-s_1} \cdots \inv_m(f)^{-s_m} \\
&= \frac1{\#G}\sum_{f:\Gal(\ol K|K)\ra G} \prod_\p w_{f,\p}(s_1,\dots,s_m),
\end{align*}
where the product is over all primes $\p$ of $K$ and where $w_{f,\p}(s_1,\dots,s_m)=\Nm(\p)^{-s_i}$ if $f$ has ramification type $[I_i,\gamma_i]$ at $\p$ and $w_{f,\p}(s_1,\dots,s_m)=1$ if $f$ is unramified or wildly ramified at $\p$.

A naive local-global principle (motivated by the case of abelian extensions of $\Q$ as explained in \cite[section 10]{wood-aws-lecture-notes}) would suggest that
\begin{equation}\label{local-global}
D(s_1,\dots,s_m) = \prod_\p D_\p(s_1,\dots,s_m),
\end{equation}
where
\[
D_\p(s_1,\dots,s_m) = \frac1{\#G}\sum_{f:\Gal(\ol{K_\p}|K_\p)\ra G} w_f(s_1,\dots,s_m)
\]
and as before $w_f(s_1,\dots,s_m)=\Nm(\p)^{-s_i}$ if $f$ has ramification type $[I_i,\gamma_i]$ and $w_f(s_1,\dots,s_m)=1$ if $f$ is unramified or wildly ramified. This local-global principle usually does not hold. However, we still assume that it holds ``on average'' in the sense that the coefficients of the Dirichlet series on both sides of (\ref{local-global}) have similar asymptotic behavior. (See \cite{bhargava-mass-formulae} and \cite[section 10]{wood-aws-lecture-notes}.)

We now compute the local factors $D_\p(s_1,\dots,s_m)$ for $\p\nmid\#G$.

\begin{lemma}\label{heuristic-local-lemma}
Consider a ramification type $[I,\gamma]=[I_i,\gamma_i]$ and assume that $\p$ does not divide the size $e=e_i$ of $I$. Consider the action of $U_e$ on the set of $G$-orbits of pairs $(I',\gamma')$ with $\#I'=e$ as in \Cref{ramification_types_def}. Let $A=A_i\subseteq U_e$ be the $U_e$-stabilizer of the $G$-orbit containing $(I,\gamma)$. Let $\Frob(\p)=\Frob_i(\p)\in U_e$ be the Frobenius automorphism for $\p$ in the field extension $K(\mu_e)|K$. We have
\[
\frac1{\#G}\cdot\#\{f:\Gal(\ol{K_\p}|K_\p)\ra G\textnormal{ of ramification type }[I,\gamma]\} \\
=
\begin{cases}
[U_e:A],& \Frob(\p)\in A,\\
0,& \textnormal{otherwise}.
\end{cases}
\]
\end{lemma}
\begin{proof}
Let $q$ be the size and let $p$ be the characteristic of the residue field~$k_\p$. The maximal tamely ramified extension of $K_\p$ is obtained by adjoining all $r$-th roots of $\pi_\p$ for all $r\geq1$ not divisible by $p$. For each $r$, choose a primitive $r$-th root of unity $\zeta_r$ and an $r$-th root $\sqrt[r]{\pi_\p}$ in such a way that $\zeta_r = \zeta_{rs}^s$ and $\sqrt[r]{\pi_\p}=\sqrt[rs]{\pi_\p}^s$ for all $r$ and $s$. The Galois group of the maximal tamely ramified extension of $K_\p$ is then topologically generated by two automorphisms $\varphi$ and $\tau$ subject to the relation $\varphi\tau\varphi^{-1}=\tau^q$. They are given by
\begin{align*}
\varphi(\zeta_r)&=\zeta_r^q,& \tau(\zeta_r)&=\zeta_r,\\
\varphi(\sqrt[r]{\pi_\p}) &= \sqrt[r]{\pi_\p},& \tau(\sqrt[r]{\pi_\p}) &= \zeta_r\sqrt[r]{\pi_\p}.
\end{align*}
The inertia group of the maximal tamely ramified extension is topologically generated by $\tau$.

The tamely ramified homomorphisms $f:\Gal(\ol{K_\p}|K_\p)\ra G$ therefore exactly correspond to pairs $\widetilde\varphi,\widetilde\tau$ of elements of $G$ that satisfy $\widetilde\varphi\widetilde\tau\widetilde\varphi^{-1}=\widetilde\tau^q$, and the inertia group $I'$ of such a homomorphism is then generated by $\widetilde\tau$. Assume the inertia group has size $e'$. Then, $\sqrt[e']{\pi_\p}$ is a uniformizer for the extension corresponding to $f$, so the inertia type $[I',\gamma']$ of $f$ is given by $\gamma'(\widetilde\tau)=\zeta_{e'}$.

Associate to the homomorphism $f$ the particular pair $t(f)=(I',\gamma')$ defined in the previous paragraph. We will show that for each $G$-orbit $O$ of pairs $(I',\gamma')$ in the equivalence class $[I,\gamma]$, we have
\[
\frac1{\#G}\cdot\#\{f:\Gal(\ol{K_\p}|K_\p)\ra G\mid t(f)\in O\} \\
=
\begin{cases}
1,& \Frob(\p)\in A,\\
0,& \textnormal{otherwise}.
\end{cases}
\]
This will then imply the claim since there are $[U_e:A]$ such $G$-orbits $O$. Because $U_e$ is abelian, every orbit has the same stabilizer $A$, so without loss of generality, it suffices to consider the $G$-orbit containing $(I,\gamma)$. Pairs $(I',\gamma')$ in this $G$-orbit bijectively correspond to elements of the conjugacy class $C$ containing $\gamma^{-1}(\zeta_e)$. (The pair $(I',\gamma')$ is determined by the generator $\gamma'^{-1}(\zeta_e)$ of $I'$.) If $C=C^q$, then for each $\widetilde\tau\in C$ there are $\#G/\#C$ elements $\widetilde\varphi\in G$ with $\widetilde\varphi\widetilde\tau\widetilde\varphi^{-1}=\widetilde\tau^q$. This means that there are $\#G/\#C$ homomorphisms $f:\Gal(\ol{K_\p}|K_\p)\ra G$ with associated pair $(I',\gamma')$, so the contribution of the entire $G$-orbit to the result is $\frac1{\#G}\cdot\#C\cdot\#G/\#C=1$. On the other hand, if $C\neq C^q$, there is no such element $\widetilde\varphi$ for any $\widetilde\tau\in C$, and therefore no such homomorphism $f$. Finally, note that the Frobenius automorphism $\Frob(\p)$ sends $\zeta_e$ to $\zeta_e^q$. Hence, it fixes the $G$-orbit containing $(I,\gamma)$ if and only if $C=C^q$.
\end{proof}

\begin{corollary}
For all $\p\nmid\#G$, we have
\[
D_\p(s_1,\dots,s_m) = 1 + \sum_{\substack{1\leq i\leq m:\\\Frob_i(\p)\in A_i}} [U_{e_i}:A_i]\Nm(\p)^{-s_i}.
\]
\end{corollary}
\begin{proof}
This follows immediately from the lemma. (The summand $1$ comes from the ``unramified'' ramification type with $I=1$.)
\end{proof}

\begin{corollary}\label{heuristic-approximation}
Let $F_i$ be the subfield of $K(\mu_{e_i})$ fixed by $A_i\subseteq U_{e_i}=\Gal(K(\mu_{e_i})|K)$. Let $\zeta_{F_i}(s)$ be its Dedekind zeta function.
\begin{enumerate}[label=\alph*)]
\item If as in \Cref{tauberian-lemma}, we write $D_\p(s_1,\dots,s_m)=b_\p+\sum_{i=1}^m c_{\p,i}\Nm(\p)^{-s_i}$ and let $D_{\p,i}(s)=b_\p+c_{\p,i}\Nm(\p)^{-s}$, then the quotient
\[
\frac{\prod_{\p:\Nm(\p)>E} D_{\p,i}(s)}{\zeta_{F_i}(s)}
\]
can be holomorphically continued to $\{\Re(s)\geq1\}$ for all $E$.
\item The quotient
\[
\frac{\prod_\p D_\p(s_1,\dots,s_m)}{\prod_{i=1}^m\zeta_{F_i}(s)}
\]
can be holomorphically continued to $\{\Re(s_1),\dots,\Re(s_m)\geq1\}$.
\end{enumerate}
\end{corollary}
\begin{proof}
For any $\p\nmid\#G$, we have
\[
D_{\p,i}(s) =
\begin{cases}
1+[U_{e_i}:A_i]\Nm(\p)^{-s}, &\Frob_i(\p)\in A_i,\\
1, &\textnormal{otherwise}.
\end{cases}
\]
Up to a holomorphic factor of absolute value $1+\O(\Nm(\p)^{-2})$ for $\Re(s)\geq1$, this is
\[
\begin{cases}
(1+\Nm(\p)^{-s})^{[U_{e_i}:A_i]}, &\Frob_i(\p)\in A_i,\\
1, &\textnormal{otherwise}.
\end{cases}
\]
Again up to a holomorphic factor of absolute value $1+\O(\Nm(\p)^{-2})$, this is the local factor in $\zeta_{F_i}(s)$ for the prime $\p$. This shows claim~a). Then, b) follows because $D_\p(s_1,\dots,s_m)/\prod_{i=1}^m D_{\p,i}(s_i)$ is holomorphic of absolute value $1+\O(\Nm(\p)^{-2})$ for $\Re(s)\geq1$.
\end{proof}

Each zeta function $\zeta_{F_i}(s)$ is holomorphic in $\{\Re(s)\geq1\}$ except for a simple pole at $s=1$. Using \Cref{tauberian-lemma} and \Cref{tauberian-boundary-lemma}, part a) of \Cref{heuristic-approximation} therefore explains \Cref{wrong-conjecture} and \Cref{wrong-boundary-conjecture} under the assumption of the local-global principle (\ref{local-global}). (Since the coefficients of $D_{\p,i}(s)$ are nonnegative, we can take $D_{\p,i}'=D_{\p,i}$ in \Cref{tauberian-lemma}.)

\section{Abelian groups}\label{abelian-section}

Assume that $G$ is abelian. Define $D(s_1,\dots,s_m)$ and $w_f$ as in the previous section.

Abelian extensions of number fields are determined by class field theory: Let $\A_K^\times = \prod'_v K_v^\times$ be the group of idèles of $K$ and let $\widehat{\A_K^\times/K^\times}$ be the profinite completion of $\A_K^\times/K^\times$. We then have an Artin reciprocity isomorphism
\[
\theta_K: \widehat{\A_K^\times/K^\times} \ra \Gal(K^\ab|K).
\]
For any prime $\p$ of $K$, we have an Artin reciprocity isomorphism
\[
\theta_{K_\p}:K_\p^\times\ra\Gal(K_\p^\ab|K_\p).
\]
The image of $\O_\p^\times$ under this isomorphism is the inertia group $I(K_\p^\ab|K_\p)$ and the image of the group $U_\p^{(1)}$ of $1$-units is the wild inertia group $I^1(K_\p^\ab|K_\p)$. Moreover, the diagram
\[
\begin{tikzcd}
\widehat{\A_K^\times/K^\times} \rar{\theta_K}[swap]{\sim} & \Gal(K^\ab|K) \\
K_\p^\times \rar{\theta_{K_\p}}[swap]{\sim} \arrow[u,hook] & \Gal(K_\p^\ab|K_\p) \arrow[u,hook]
\end{tikzcd}
\]
commutes, where the vertical arrows are the natural inclusion maps.

Using the Artin reciprocity isomorphisms, we can identify the continuous homomorphisms $\Gal(\ol K|K)\ra G$ with the continuous homomorphisms $\A_K^\times/K^\times\ra G$ and the continuous homomorphisms $\Gal(\ol{K_\p}|K_\p)\ra G$ with the continuous homomorphisms $K_\p^\times\ra G$.

As we will show in \Cref{abelian_ram_type}, the ramification type of a homomorphism $\Gal(\ol{K_\p}|K_\p)\ra G$ can be determined from the restriction of the corresponding homomorphism $K_\p^\times\ra G$ to $\O_\p^\times$. We can thus determine the ramification type of $\p$ for a homomorphism $f:\widehat{\A_K^\times/K^\times}\ra G$ from the restriction $f:\O_\p^\times\ra G$.

Because $K_\p^\times = \O_\p^\times\times\pi_\p^\Z$, each homomorphism $\O_\p^\times\ra G$ is the restriction of exactly $\#G$ homomorphisms $K_\p^\times\ra G$.

For $K=\Q$, the product decomposition $\widehat{\A_\Q^\times/\Q^\times} = \prod_p \Z_p^\times$ then implies the local-global principle as described in (\ref{local-global}), so the discussion in \Cref{heuristic-section} immediately shows \Cref{wrong-conjecture} and \Cref{wrong-boundary-conjecture}.

However, for other number fields $K$, the local-global principle is slightly obstructed by the class and unit groups of $K$. To solve this issue, we follow the same strategy as Wright in \cite{wright-counting-abelian-extensions} and Wood in \cite{wood-probabilities-of-local-behaviors}.

For a finite set $S$ of places of $K$ containing all archimedean places, consider the group
\[
\prod_{v\notin S}\O_v^\times\setminus \prod_{v\notin S}\nolimits' K_v^\times/K^\times.
\]
We can fix a set $S$ such that this group is trivial. (Pick primes generating the ideal class group of $K$ and let $S$ consist of these primes and all archimedean places.)
This means that continuous homomorphisms $\A_K^\times/K^\times\ra G$ exactly correspond to collections $(f_v)_v$ of continuous homomorphisms with $f_v:K_v^\times\ra G$ for $v\in S$ and $f_v:\O_v^\times\ra G$ for $v\notin S$ such that almost all $f_v$ are trivial (which for $v\notin S$ means unramified) and $\prod_v f_v(x)=1$ for all $x\in\O_S^\times$. (Here, $\O_S^\times$ denotes the group of $S$-units.)

In other words, we should consider only those tuples $(f_v)_v$ such that the restriction $\prod_v f_v|_{\O_S^\times} \in \Hom(\O_S^\times,G)$ is trivial. Since the group $\O_S^\times$ of $S$-units in $K$ is finitely generated, the abelian group $\Hom(\O_S^\times,G)$ (which obstructs the local-global principle) is luckily finite. We can therefore write
\begin{align*}
D(s_1,\dots,s_m)
&= \frac{1}{\#G} \cdot \sum_{f:\Gal(\ol K|K)\ra G} \prod_\p w_{f,\p}(s_1,\dots,s_m) \\
&= \frac{1}{\#G \cdot \#\Hom(\O_S^\times,G)} \cdot \sum_{(f_v)_v} \sum_\rho \rho\big(\prod_v f_v|_{\O_S^\times}\big) \prod_\p w_{f_\p}(s_1,\dots,s_m),
\end{align*}
with the inner sum ranging over all characters $\rho$ of $\Hom(\O_S^\times,G)$. This can be written as
\begin{align*}
D(s_1,\dots,s_m)
&= \frac{1}{\#G \cdot \#\Hom(\O_S^\times,G)} \cdot \sum_\rho D_\rho(s_1,\dots,s_m),
\end{align*}
where
\begin{align*}
D_\rho(s_1,\dots,s_m)
&= \sum_{(f_v)_v} \prod_v\rho(f_v|_{\O_S^\times}) w_{f_v}(s_1,\dots,s_m),
\end{align*}
with the sum ranging over tuples $(f_v)_v$ of maps $f_v:K_v^\times\ra G$ for $v\in S$ and $f_v:\O_v^\times\ra G$ for $v\notin S$. (We define $w_{f_v}$ in terms of the ramification type of $f_v$ as before. If $v$ is archimedean, we let $w_{f_v}=1$.) We can rewrite
\begin{align*}
D_\rho(s_1,\dots,s_m)
&= \prod_v D_{\rho,v}(s_1,\dots,s_m)
\end{align*}
where
\[
D_{\rho,v}(s_1,\dots,s_m)
= \sum_f \rho(f|_{\O_S^\times}) w_f(s_1,\dots,s_m),
\]
with the sum ranging over $f:K_v^\times\ra G$ for $v\in S$ and over $f:\O_v^\times\ra G$ for $v\notin S$.

We now compute the local factors $D_{\rho,\p}(s_1,\dots,s_m)$ for $\p\nmid\#G$.

\begin{lemma}\label{abelian_ram_type}
Let $f:K_\p^\times\ra G$ be a continuous homomorphism. Denote the size of the residue field $k_\p$ by $q$.
\begin{enumerate}[label=\alph*)]
\item The homomorphism $f$ is tamely ramified if and only if its kernel contains $U_\p^{(1)}$.
\item Assume $f$ is tamely ramified. Let $I\subseteq G$ be the image of $\O_\p^\times$ and denote the size of $I$ by $e$. We then have $e\mid q-1$. Let $\mu_{\p,e}$ be the group of $e$-th roots of unity in $k_\p$. Define $s:k_\p^\times\ra\mu_{\p,e}$ by $x\mapsto x^{-(q-1)/e}$. Then, the restriction of $f$ to $\O_\p^\times$ factors as $f(x) = c(s(x\bmod\p))$ for some isomorphism $c:\mu_{\p,e}\ra I$. Pick a prime $\P$ of $K(\mu_e)$ dividing~$\p$. Define $\gamma:I\ra\mu_e$ by $\gamma(g)\equiv c^{-1}(g)\mod\P$. Then, $f$ has ramification type $[I,\gamma]$.
\end{enumerate}
\end{lemma}
\begin{proof}
A homomorphism $\Gal(K^\ab|K)\ra G$ is tamely ramified if and only if its kernel contains the wild ramification group. This implies a).

For b), note first that $I$ is the inertia group of $f$. The surjection $f:\O_\p^\times\ra I$ must factor through $\O_\p^\times/U_\p^{(1)}=k_\p^\times$ according to a). In particular, $q-1$ must be divisible by $e$. The kernel of $s$ (consisting of the $e$-th powers in $k_\p^\times$) is the unique subgroup of $k_\p^\times$ of index $e$. This implies that the restriction of $f$ to $\O_\p^\times$ factors as $f(x) = c(s(x\bmod\p))$ as claimed. Extend the normalized valuation on $K_\p$ to $K_\p^\ab$. Let $\pi$ be any element of $K_\p^\ab$ of valuation $1/e$. The homomorphism $\beta:I(K_\p^\ab|K_\p)\ra k_\p^\times$ defined by $\sigma\mapsto \sigma(\pi)/\pi\bmod\p$ is independent of the choice of $\pi$. Taking $\pi=\sqrt[e]{\pi_\p}$ for a uniformizer $\pi_\p$ of $K_\p$, we see that $\beta(\sigma)\equiv\sigma(\sqrt[e]{\pi_\p})/\pi_\p$. If $\sigma\in I(K_\p^\ab|K_\p)$ corresponds to $x\in\O_\p^\times$ under the Artin reciprocity map, then $\beta(\sigma)$ is the degree $e$ Hilbert symbol $\left(\frac{x,\pi_\p}{\p}\right)$. (See \cite[Proposition V.3.1]{neukirch-algebraic-number-theory}.) This Hilbert symbol can be computed as $\beta(\sigma)\equiv\left(\frac{x,\pi_\p}{\p}\right)\equiv x^{-(q-1)/e}\mod\p$. (See \cite[Proposition V.3.4]{neukirch-algebraic-number-theory}.) It follows that the map $\gamma$ defined in the statement of the lemma agrees with the map $\gamma$ defined in \Cref{ramification_type_def}. Hence, $f$ indeed has ramification type $[I,\gamma]$.
\end{proof}

\begin{lemma}\label{abelian-local-lemma}
Consider a ramification type $[I,\gamma]=[I_i,\gamma_i]$ and assume that $\p$ does not divide the size $e=e_i$ of~$I$. Let $F=F_i=K(\mu_e,\sqrt[e]{\O_S^\times})$ be the finite abelian extension of~$K(\mu_e)$ generated by the $e$-th roots of all elements of $\O_S^\times$. Define the injective homomorphism $\alpha=\alpha_i:\Gal(F|K(\mu_e))\ra\Hom(\O_S^\times,G)$ as follows: For any $\sigma\in\Gal(F|K(\mu_e))$ and $x\in\O_S^\times$, write $\sigma(\sqrt[e]x) = \zeta\sqrt[e]x$ with $\zeta\in\mu_e$. Then, define $\alpha(\sigma)(x) = \gamma^{-1}(\zeta^{-1})$.

Let $\Frob(\p)=\Frob_i(\p)\in\Gal(K(\mu_e)|K)$ be the Frobenius automorphism for $\p$. For any prime $\P|\p$ of $K(\mu_e)$, let $\Frob(\P)=\Frob_i(\P)\in\Gal(F|K(\mu_e))$ be its Frobenius automorphism. For any character $\rho$ of $\Hom(\O_S^\times,G)$, we then have
\[
\sum_{\substack{f:\O_\p^\times\ra G\\\textnormal{of ram. type }[I,\gamma]}} \rho(f|_{\O_S^\times})
=
\begin{cases}
\sum_{\P\mid\p} \rho(\alpha(\Frob(\P))),& \Frob(\p)=\id,\\
0,& \textnormal{otherwise}.
\end{cases}
\]
\end{lemma}

When comparing this statement with \Cref{heuristic-local-lemma}, note that $G$ acts trivially on pairs $(I,\gamma)$ and that the stabilizers $A$ are trivial.

\begin{proof}
Let $q$ be the size of the residue field $k_\p$. Again, $\Frob(\p)$ is the automorphism sending an $e$-th root of unity $\zeta_e$ to $\zeta_e^q$. It is the identity if and only if $e$ divides $q-1$. If $\Frob(\p)\neq\id$, we conclude that there is no $f:\O_\p^\times\ra G$ of ramification type $[I,\gamma]$.

Assume that $\Frob(\p)=\id$, so $e\mid q-1$ and $\p$ splits completely in $K(\mu_e)$. For each prime $\P$ of $K(\mu_e)$ dividing $\p$, there is exactly one isomorphism $c:\mu_{\p,e}\ra I$ such that $\gamma(g)\equiv c^{-1}(g)\mod\P$ for all $g\in I$. Two different primes $\P_1,\P_2\mid\p$ give rise to two different isomorphisms $c_1,c_2$. We will show that if $f:\O_\p^\times\ra G$ is given by $f(x)=c(s(x\bmod\p))$ as in \Cref{abelian_ram_type}, then $f|_{\O_S^\times} = \alpha(\Frob(\P))$. This then implies the claim.

To compute $f|_{\O_S^\times}$, note that for any $x\in\O_\p^\times$, we have
\begin{align*}
\gamma(f(x))
&\equiv s(x\bmod\p)
\equiv x^{-(q-1)/e}
\equiv \sqrt[e]{x}/\sqrt[e]{x}^q \\
&\equiv \sqrt[e]{x}/\Frob(\P)(\sqrt[e]{x})
\equiv \gamma(\alpha(\Frob(\P))(x)) \mod \P.
\end{align*}
This implies that indeed $f(x) = \alpha(\Frob(\P))(x)$.
\end{proof}

\begin{corollary}
For all $\p\nmid\#G$, we have
\[
D_{\rho,\p}(s_1,\dots,s_m) = 1 + \sum_{\substack{1\leq i\leq m:\\\Frob_i(\p)=\id}} \sum_{\substack{\P\mid\p\\\textnormal{prime of}\\K(\mu_{e_i})}} \rho(\alpha_i(\Frob_i(\P))).
\]
\end{corollary}

\begin{corollary}
For any $i$, note that the composition $\rho\circ\alpha_i$ is a character of $\Gal(F_i|K(\mu_{e_i}))$. Let $L(F_i|K(\mu_{e_i}),\rho\circ\alpha_i,s)$ be the corresponding Artin L-function.
\begin{enumerate}[label=\alph*)]
\item If as in \Cref{tauberian-lemma}, we write $D_{\rho,\p}(s_1,\dots,s_m)=b_{\rho,\p}+\sum_{i=1}^m c_{\rho,\p,i}\Nm(\p)^{-s_i}$ and let $D_{\rho,\p,i}(s)=b_{\rho,\p}+c_{\rho,\p,i}(s)\Nm(\p)^{-s}$, then the quotient
\[
\frac{\prod_{\p:\Nm(\p)>E} D_{\rho,\p,i}(s)}{L(F_i|K(\mu_{e_i}),\rho\circ\alpha_i,s)}
\]
can be holomorphically continued to $\{\Re(s)\geq1\}$ for all $E$.
\item The quotient
\[
\frac{D_\rho(s_1,\dots,s_m)}{\prod_{i=1}^m L(F_i|K(\mu_{e_i}),\rho\circ\alpha_{e_i},s_i)}
\]
can be holomorphically continued to $\{\Re(s_1),\dots,\Re(s_m)\geq1\}$.
\end{enumerate}
\end{corollary}
\begin{proof}
For any $\p\nmid\#G$, we have
\[
D_{\rho,\p,i}(s) =
\begin{cases}
1+\sum_{\substack{\P\mid\p\\\textnormal{prime of}\\K(\mu_{e_i})}} \rho(\alpha_i(\Frob_i(\P))) \cdot \Nm(\p)^{-s}, &\Frob_i(\p)=\id,\\
1, &\textnormal{otherwise}.
\end{cases}
\]
Up to a holomorphic factor of absolute value $1+\O(\Nm(\p)^{-2})$ for $\Re(s)\geq1$, this is
\[
\begin{cases}
\prod_{\substack{\P\mid\p\\\textnormal{prime of}\\K(\mu_e)}}(1+\rho(\alpha_i(\Frob_i(\P)))\Nm(\p)^{-s}), &\Frob_i(\p)=\id,\\
1, &\textnormal{otherwise}.
\end{cases}
\]
Again up to a holomorphic factor of absolute value $1+\O(\Nm(\p)^{-2})$, this is the local factor in $L(F_i|K(\mu_{e_i}),\rho\circ\alpha_{e_i},s)$ for the prime $\p$. This shows claim~a). Then, b) follows as before.
\end{proof}

\begin{theorem}
\Cref{wrong-conjecture} and \Cref{wrong-boundary-conjecture} hold for any abelian group~$G$.
\end{theorem}
\begin{proof}
An Artin L-function $L(F_i|K(\mu_{e_i}),\rho\circ\alpha_i,s)$ is holomorphic in $\{\Re(s)\geq1\}$, except for a simple pole with positive residue at $s=1$ if $\rho\circ\alpha_i$ is the trivial character of $\Gal(F_i|K(\mu_{e_i}))$.

Write
\[
D(s_1,\dots,s_m)
= \sum_{x_1,\dots,x_m\geq1}a_{x_1,\dots,x_m}x_1^{-s_1}\cdots x_m^{-s_m}.
\]
Then, \Cref{tauberian-lemma} (using $D_{\rho,\p,i}$ as $D_{\p,i}$ and $D_{\triv,\p,i}$ as $D_{\p,i}'$) shows that
\[
\sum_{x_1\leq X_1,\dots,x_m\leq X_m} a_{x_1,\dots,x_m}
= \sum_{\rho} C_\rho\cdot X_1\cdots X_m + o(X_1\cdots X_m),
\]
where $C_\rho=0$ if $\rho\circ\alpha_{e_i}$ is nontrivial for some $i$ and $C_\rho\neq0$ otherwise.

For \Cref{wrong-conjecture}, it remains to show that $\sum_\rho C_\rho > 0$, or in other words to show the lower bound $\sum_{x_1\leq X_1,\dots,x_m\leq X_m}a_{x_1,\dots,x_m}\gg X_1\cdots X_m$. In fact, one can show the analogous lower bound when only considering extensions that are unramified at all primes $\p\mid\#G$: In the associated Dirichlet series (omitting summands corresponding to ramification at a prime $\p\mid\#G$), if all $\rho\circ\alpha_{e_i}$ are trivial, then all summands in $\sum_{\P|\p}\rho(\alpha_i(\Frob_1(\P)))$ are $1$, so all summands in the Dirichlet series are positive. This makes all residues positive, so the product $C_\rho'$ of residues is positive whenever it is nonzero (for example when $\rho$ is the trivial character).

\Cref{wrong-boundary-conjecture} follows similarly from \Cref{tauberian-boundary-lemma}. The corresponding constant factor $C_{\triv}$ for the trivial character $\rho$ could be zero in this case. Then, the coefficients in $D_{\triv}(s_1,\dots,s_m)$ with fixed $x_i$ for $i\in T$ are all zero according to \Cref{tauberian-boundary-lemma}. As the coefficients in $D_{\triv}$ provide an upper bound for the coefficients in $D$, it follows in that case that the sum on the left-hand side of \Cref{wrong-boundary-conjecture} is empty, so the result still holds, with $C=0$.
\end{proof}

\begin{remark}
\Cref{wrong-conjecture} and \Cref{wrong-boundary-conjecture} also hold for any abelian group~$G$ if we only count field $G$-extensions, rather than étale $G$-extensions.
\end{remark}
\begin{proof}
Any étale $G$-extension is induced by a field $H$-extension $L'$ for some subgroup $H$ of $G$. The invariant of the induced $G$-extension $L$ corresponding to a ramification type $[I_i,\gamma_i]$ with $I_i\nsubseteq H$ is $1$. On the other hand, if $I_i\subseteq H$, then $[I_i,\gamma_i]$ is also a ramification type for the group $H$. The invariant of $L$ corresponding to $[I_i,\gamma_i]$ is the invariant of $L'$ corresponding to $[I_i,\gamma_i]$. Applying \Cref{wrong-conjecture} and \Cref{wrong-boundary-conjecture} to $H$, the claim then follows from inclusion--exclusion over all subgroups $H$ of~$G$.
\end{proof}

\section{Relationship with Malle's conjecture}\label{relationship-with-malle-section}

Ignoring wildly ramified primes, \Cref{wrong-conjecture} and \Cref{wrong-boundary-conjecture} imply following variants of Malle's conjecture:

\begin{remark}
Assume the group $G$ satisfies \Cref{wrong-conjecture}. Let $h_1,\dots,h_m$ be real numbers and $a=\min(h_1,\dots,h_m)$. Let exactly $b$ of the numbers $h_1,\dots,h_m$ be equal to $a$. Let
\[
N(X) = 
\sum_{\substack{\textnormal{$G$-ext. $L|K$}:\\\prod_{i=1}^m \inv_i(L)^{h_i} \leq X}} \frac1{\#\Aut(L)}
\]
If $a>0$, then for large $X$, we have $N(X) \asymp X^{1/a} (\log X)^{b-1}$.

Otherwise, for sufficiently large $X$, we have $N(X) = \infty$.
\end{remark}
\begin{proof}
The lower bound follows by approximating the region
\[
R=\{(x_1,\dots,x_m)\in\R_{\geq1}^m\mid x_1^{h_1}\cdots x_m^{h_m} \leq X\}
\]
by the union of boxes of the form $\prod_{i=1}^m(\delta^{-n_i+1},\delta^{-n_i}]$ and applying \Cref{bold-conjecture} (which follows from \Cref{wrong-conjecture}).

The upper bound follows by embedding $R$ into a union of boxes of the form $\prod_{i=1}^m[1,X_i]$.
\end{proof}

\begin{remark}
Under the same assumptions as in the previous remark, if we assume that $G$ satisfies the stronger \Cref{wrong-boundary-conjecture}, then in the case $a>0$, there is a constant $C>0$ such that for $X$ going to $\infty$, we have $N(X) \sim CX^{1/a}(\log X)^{b-1}$.
\end{remark}
\begin{proof}
Approximate the region $R$ by the union of boxes of the form
\[
\prod_{i\in T}\{x_i\}\times\prod_{i\notin T}(\delta^{-n_i+1}q,\delta^{-n_i}q]
\]
with $x_i\leq q$ for suitably large $q$ and small $\delta$ and subsets $T$ of $\{1,\dots,m\}$.
\end{proof}

\section{Fine conductors}\label{fine-conductor-section}

Let $K$ be a nonarchimedean local field.

To a $G$-extension $L|K$, or a continuous homomorphism $f:\Gal(K^\sep|K)\ra G$, one famously associates the \emph{Artin conductors} with respect to (finite-dimensional) complex representations $\rho$ of $G$. It is interesting to note that $\Gal(\ol\Q|\Q)$-conjugate representations produce the same conductor. Hence, Artin conductors can really be thought of as associated to representations defined over $\Q$.

Can we make use of all representations defined over the base field $K$?

In \cite[section 3]{wood-yasuda-1}, Wood and Yasuda associate to any continuous homomorphism $f:\Gal(K^\sep|K)\ra G$ the \emph{weights} $w(f,\rho)$ with respect to representations $\rho$ of $G$ defined over $K$. The weights in general capture more information than Artin conductors. However, as illustrated by the following example, the weight functions associated to the nontrivial irreducible representations are in general still not linearly independent (at least if we ignore wildly ramified primes).
\begin{example}\label{wood-yasuda-linearly-dependent}
Consider the cyclic group $G=C_5$ and any nonarchimedean local field $K$ of characteristic zero containing $\zeta_5$. There are five irreducible representations $\triv,\rho,\rho^2,\rho^3,\rho^4$ of $G$ defined over~$K$. According to \cite[Corollary 4.6]{wood-yasuda-1}, for any tamely ramified homomorphism $f$, the sums
\[
w(f,\rho)+w(f,\rho^4)\qquad\textnormal{ and }\qquad w(f,\rho^2)+w(f,\rho^3)
\]
are the valuation of the Artin conductor for $\rho$ and for $\rho^2$, respectively. But $\rho$ and $\rho^2$ are $\Gal(\ol\Q|\Q)$-conjugate representations, so their Artin conductors agree. Hence, the weights $w(f,\rho^i)$ satisfy the linear relation
\[
w(f,\rho)+w(f,\rho^4) = w(f,\rho^2)+w(f,\rho^3)
\]
for all tamely ramified homomorphisms $f$.
\end{example}

In this section, we associate conductors to representations defined over $K$ such that the conductors associated to the nontrivial irreducible representations are linearly independent.

We assume from now on that the characteristic of $K$ does not divide the order of $G$. The representation theory of $G$ over $\C$ is then identical to the representation theory over $K^\sep$. Furthermore, every representation defined over $\Q$ naturally corresponds to a representation defined over $K$.

We endow $K$ with the normalized valuation and extend this valuation to $K^\sep$.

We denote the inertia group of $K^\sep|K$ by $I\subseteq\Gal(K^\sep|K)$ and the higher ramification groups of $K^\sep|K$ in upper numbering by $I^t\subseteq I(K^\sep|K)$.

Let $S_{K,G}$ be the set of continuous homomorphisms $f:\Gal(K^\sep|K)\ra G$. For any such $f$, we will denote by $L_f$ the field fixed by $\ker(f)$. It is a Galois extension of $K$ with Galois group $\im(f)\subseteq G$. We will denote the prime of $K^\sep$ by $\ol\p$. We denote the size of $f(I)$ by $e(f)$ and the size of $f(I)/f(I^1)$ by $e_0(f)$.

We denote the inner product of two characters $\psi_1,\psi_2$ of a group $H$ by $\langle\psi_1,\psi_2\rangle_H$.

Artin conductors can be defined as follows (see for example \cite[section II.10]{neukirch-algebraic-number-theory}).
\begin{definition}
The \emph{Artin conductor} of $f\in S_{K,G}$ with respect to the character $\psi$ of a complex representation of $G$ is
\begin{align*}
\fartin(f,\psi)
&= \int_0^\infty \left(\psi(\id) - \langle\psi,\triv\rangle_{f(I^t)}\right)\dd t.
\end{align*}
\end{definition}

For any integer $n\geq1$ relatively prime to the residue field characteristic of $K$, we consider the surjective homomorphism $\gamma_n:I/I^1\ra\mu_n$ defined by $\gamma_n(\sigma)\equiv \sigma(x)/x \mod\ol\p$, where $x$ is an element of $K^\sep$ of valuation $1/n$. (This is independent of the choice of $x$. Furthermore, $\gamma_{nm}^m = \gamma_n$.)

We now modify the definition of Artin conductors to take into account the additional information on tame ramification provided by the isomorphism $\gamma_f:f(I)/f(I^1)\ra\mu_{e_0(f)}$ given by $f(\sigma)\mapsto\gamma_{e_0(f)}(\sigma)$. (This is well-defined because the subfield of $K^\sep$ fixed by $\ker(f)$ contains an element $x$ of valuation $1/e_0(f)$.)

If $f$ is a tamely ramified homomorphism and corresponds to the $G$-extension $L|K$, then this map $\gamma_f$ agrees with the map $\gamma$ defined in \cref{invariants-section}.

We will also write $\gamma_f$ for the composition $f(I)\ra f(I)/f(I^1)\ra\mu_{e_0(f)}$ and interpret it as a representation of $f(I)$ defined over $K^\sep$. We write $\gamma_f^d$ for the $d$-th power of $\gamma_f$.

Let $\ol I^t=I^1$ for $0\leq t<1$ and $\ol I^t = I^t$ for $t\geq1$. Denote the Euler phi function by $\varphi$.

\begin{definition}
The \emph{fine Artin conductor} of $f\in S_{K,G}$ with respect to the character $\psi$ of a representation of $G$ defined over $K$ is
\begin{align*}
\ffine(f,\psi)
&= \int_0^\infty\left(\psi(\id)-\langle\psi,\triv\rangle_{f(\ol I^t)}\right)\dd t
+ \sum_{\substack{d\mid e_0(f)\\d\neq e_0(f)}} \varphi\left(\frac{e_0(f)}{d}\right)\langle\psi,\gamma_f^d\rangle_{f(I)}.
\end{align*}
\end{definition}

\begin{remark}\label{fine_tame_remark}\
\begin{enumerate}[label=\alph*)]
\item If $f$ is unramified, then
\[
\fartin(f,\psi)=\ffine(f,\psi)=0.
\]
\item If $f$ is tamely ramified, then
\[
\fartin(f,\psi) = \psi(\id) - \langle\psi,\triv\rangle_{f(I)}
\]
and
\[
\ffine(f,\psi) = \sum_{\substack{d\mid e(f)\\d\neq e(f)}} \varphi\left(\frac{e(f)}{d}\right)\langle\psi,\gamma_f^d\rangle_{f(I)}
\]
and the Wood--Yasuda weight is
\[
w(f,\psi) = \psi(\id) - \sum_{k=1}^e \frac{k}{e} \langle\psi,\gamma_f^k\rangle_{f(I)}.
\]
\end{enumerate}
\end{remark}
\begin{proof}
In case a), we have $f(I)=1$ and in case b), we have $f(I^1)=1$. The statement about $w(f,\psi)$ follows from \cite[remark after Lemma 4.3]{wood-yasuda-1}.
\end{proof}

\begin{lemma}
The conductor $\ffine(f,\psi)$ is a nonnegative integer.
\end{lemma}
\begin{proof}
It is well-known that $\fartin(f,\psi)$ is an integer. Since $I^t=I$ and $\ol I^t=I^1$ for $0\leq t<1$ and $I^t=\ol I^t$ for $t\geq1$, we have
\begin{equation}\label{fine_minus_artin}
\ffine(f,\psi) - \fartin(f,\psi) 
= \langle\psi,\triv\rangle_{f(I)} - \langle\psi,\triv\rangle_{f(I^1)} + \sum_{\substack{d\mid e_0(f)\\d\neq e_0(f)}}\varphi\left(\frac{e_0(f)}{d}\right)\langle\psi,\gamma_f^d\rangle_{f(I)}.
\end{equation}
This is clearly an integer, so $\ffine(f,\psi)$ is an integer, too.

Moreover, the integrand and all summands in the definition of $\ffine(f,\psi)$ are nonnegative.
\end{proof}

\begin{theorem}\label{fine_same}
If $\psi$ is the character of a representation defined over $\Q$, then
\[
\ffine(f,\psi) = \fartin(f,\psi).
\]
\end{theorem}
\begin{proof}
Let $e_0=e_0(f)$. Since the map $\gamma_f$ has image $\mu_{e_0}$ and $\psi$ is invariant under the action of $\Gal(\Q(\mu_{e_0})|\Q)\cong(\Z/e_0\Z)^\times$, we have
\[
\langle\psi,\gamma_f^k\rangle_{f(I)} = \langle\psi,\gamma_f^{\gcd(k,e_0)}\rangle_{f(I)}
\]
for all $k\in\Z/e_0\Z$. Hence, according to (\ref{fine_minus_artin}), we have
\[
\ffine(f,\psi)-\fartin(f,\psi)
= \langle\psi,\triv\rangle_{f(I)} - \langle\psi,\triv\rangle_{f(I^1)} + \sum_{0\neq k\in\Z/e_0\Z}\langle\psi,\gamma_f^k\rangle_{f(I)}.
\]
Recall that $\gamma_f$ is an isomorphism $f(I)/f(I^1) \ra \mu_{e_0}$. Hence, the representations $\gamma_f^k$ of $f(I)/f(I^1)$ for $0\neq k\in\Z/e_0\Z$ are exactly the nontrivial representations of $f(I)/f(I^1)$. Thus,
\[
\triv_{f(I)} + \sum_{0\neq k\in\Z/e_0\Z}\gamma_f^k
\]
is the character of the regular representation of $f(I)/f(I^1)$. Hence,
\[
\sum_{0\neq k\in\Z/e_0\Z}\langle\psi,\gamma_f^k\rangle_{f(I)}
= \langle\psi,\triv\rangle_{f(I^1)} - \langle\psi,\triv\rangle_{f(I)},
\]
so indeed $\ffine(f,\psi)-\fartin(f,\psi)=0$.
\end{proof}

Like Artin conductors, fine conductors satisfy the following properties.

\begin{theorem}\label{fine_properties}
We have:
\begin{enumerate}[label=\alph*)]
\item\label{prop_triv} For the trivial representation,
\[\ffine(f,\textnormal{triv}) = 0.
\]
\item\label{prop_sum} For the direct sum of two representations,
\[\ffine(f,\psi_1+\psi_2) = \ffine(f,\psi_1) + \ffine(f,\psi_2).
\]
\item\label{prop_conj} For any $g\in G$,
\[
\ffine(gfg^{-1},\psi)=\ffine(f,\psi).
\]
\item\label{prop_hom} Let $\phi:G_1\ra G_2$ be a group homomorphism. For any $f\in S_{K,G_1}$ and the character $\psi$ of any representation of $G_2$ defined over $K$, we have
\[
\ffine(\phi\circ f,\psi) = \ffine(f,\psi\circ\phi).
\]
\end{enumerate}
\end{theorem}

\begin{remark}
Properties \ref{prop_triv}, \ref{prop_sum}, \ref{prop_hom}, together with the fact that the definition of $\ffine$ only takes into account the ramification groups, show that $\ffine$ gives rise to a complete counting system in the sense of \cite[Definition 2.2]{wood-yasuda-1}.
\end{remark}
\begin{remark}
Property \ref{prop_conj} shows that $\ffine(f,\psi)$ can be interpreted as an invariant $\ffine(L|K,\psi)$ of the $G$-extension $L$ of $K$ corresponding to $f$. For any $G$-extension $L|K$ of number fields and the character of any representation of $G$ defined over $K$, we can then also define the global fine Artin conductor
\[
\condfine(L|K,\psi) = \prod_{\p\textnormal{ prime of }K} \p^{\ffine(L_\p|K_\p,\psi)}
\]
in analogy with the global Artin conductor
\[
\condartin(L|K,\psi) = \prod_{\p\textnormal{ prime of }K} \p^{\fartin(L_\p|K_\p,\psi)}.
\]
\end{remark}
\begin{remark}
Property \ref{prop_hom} shows that for the character $\psi$ of a representation $\rho:G\ra\GL_n(K)$, the number  $\ffine(f,\psi)$ can be interpreted as an invariant $\ffine(\tau)$ of the Galois representation $\tau=\rho\circ f:\Gal(K^\sep|K)\ra\GL_n(K)$.
\end{remark}

\begin{proof}[Proof of \Cref{fine_properties}]\
\begin{enumerate}[label=\alph*)]
\item For $\psi=\triv$, the integrand in the definition of $\ffine$ is
\[
\psi(\id)-\langle\psi,\triv\rangle_{\ol I^t} = 1-1 = 0
\]
and the summands $\varphi(e_0/d)\langle\psi,\gamma_f^d\rangle_{f(I)}$ are $0$ because the one-dimen\-sio\-nal representation $\gamma_f^d$ is nontrivial. (Its image is $\mu_{e_0(f)/d}$.)
\item This is clear from the definition.
\item This is clear from the definition. Alternatively, it follows from \ref{prop_hom}, letting $\phi:G\ra G$ be the conjugation by $g$ map.
\item First, note that $e_0(\phi\circ f)$ divides $e_0(f)$ because we have a surjection $\phi:f(I)/f(I^1)\ra\phi(f(I))/\phi(f(I^1))$.

Secondly, for any $g_1\in f(I)$, it is easy to check that
\begin{equation}\label{claim0}
\gamma_{\phi\circ f}(\phi(g_1)) = \gamma_f(g_1)^{e_0(f)/e_0(\phi\circ f)}.
\end{equation}
Denote the surjection $\phi:f(I)\ra\phi(f(I))$ by $\widetilde\phi$. Its kernel has size $e(f)/e(\phi\circ f)$. We next show that for any $d\mid e_0(f)$, we have
\begin{equation}\label{claim1}
\sum_{g_1\in\ker(\widetilde\phi)}\gamma_f(g_1)^{-d}
=
\begin{cases}
e(f)/e(\phi\circ f),& e_0(f)/e_0(\phi\circ f) \mid d,\\
0,&\textnormal{otherwise}.
\end{cases}
\end{equation}
First, recall that $\gamma_f(g_1)$ only depends on the residue class of $g_1$ modulo $f(I^1)$ and that it in fact corresponds to an isomorphism $\gamma_f:f(I)/f(I^1)\ra\mu_{e_0(f)}$. If we let $S$ be the image of $\ker(\widetilde\phi)$ under the quotient map $f(I)\ra f(I)/f(I^1)$, then
\[
\sum_{g_1\in\ker(\widetilde\phi)}\gamma_f(g_1)^{-d}
= \#f(I^1) \cdot \sum_{r_1\in S}\gamma_f(r_1)^{-d}.
\]
The set $S$ is the kernel of the surjection $\phi:f(I)/f(I^1)\ra\phi(f(I))/\phi(f(I^1))$, so its size is $e_0(f)/e_0(\phi\circ f)$. Since $f(I)/f(I^1)\cong\mu_{e_0(f)}$ is cyclic, the size of $S$ fully characterizes the subgroup $S$. The kernel of $\gamma_f^{-d}$ is the cyclic subgroup of $f(I)/f(I^1)$ of size $d$. If $S$ is contained in the kernel (meaning $e_0(f)/e_0(\phi\circ f)\mid d$), then we conclude that the sum is
\[
\sum_{g_1\in\ker(\widetilde\phi)}\gamma_f(g_1)^{-d} = \#\ker(\widetilde\phi) = e(f)/e(\phi\circ f).
\]
Otherwise, the sum is $0$. This concludes the proof of (\ref{claim1}).

Together with (\ref{claim0}), it follows that for any $g_2\in\phi(f(I))$, we have
\[
\sum_{g_1\in\widetilde\phi^{-1}(g_2)}\gamma_f(g_1)^{-d} \\
=
\begin{cases}
\frac{e(f)}{e(\phi\circ f)}\cdot\gamma_{\phi\circ f}(g_2)^{-d'},& d=\frac{e_0(f)}{e_0(\phi\circ f)}\cdot d'\textnormal{ with }d'\in\Z,\\
0, &\textnormal{otherwise}.
\end{cases}
\]
It follows that the sums in the definitions of $\ffine(f,\psi\circ\phi)$ and $\ffine(\phi\circ f,\psi)$ agree:
\begin{align*}
&\sum_{\substack{d\mid e_0(f)\\d\neq e_0(f)}}\varphi\left(\frac{e_0(f)}{d}\right)\langle\psi\circ\phi,\gamma_f^d\rangle_{f(I)} \\
&=\sum_{\substack{d\mid e_0(f)\\d\neq e_0(f)}}\varphi\left(\frac{e_0(f)}{d}\right) \frac1{e(f)}\sum_{g_1\in f(I)}\psi(\phi(g_1))\gamma_f(g_1)^{-d} \\
&=\sum_{\substack{d\mid e_0(f)\\d\neq e_0(f)}}\varphi\left(\frac{e_0(f)}{d}\right) \frac1{e(f)}\sum_{g_2\in \phi(f(I))}\psi(g_2)\sum_{g_1\in\widetilde\phi^{-1}(g_2)}\gamma_f(g_1)^{-d} \\
&=\sum_{\substack{d'\mid e_0(\phi\circ f)\\d'\neq e_0(\phi\circ f)}}\varphi\left(\frac{e_0(\phi\circ f)}{d'}\right) \frac1{e(\phi\circ f)}\sum_{g_2\in \phi(f(I))}\psi(g_2)\gamma_{\phi\circ f}(g_2)^{-d'} \\
&=\sum_{\substack{d'\mid e_0(\phi\circ f)\\d'\neq e_0(\phi\circ f)}}\varphi\left(\frac{e_0(\phi\circ f)}{d'}\right) \langle\psi,\gamma_{\phi\circ f}^{d'}\rangle_{\phi(f(I))}.
\end{align*}
The integrals in the definitions of $\ffine(f,\psi\circ\phi)$ and $\ffine(\phi\circ f,\psi)$ are easily seen to agree.
\qedhere
\end{enumerate}
\end{proof}

Consider the nontrivial ramification types $[I_1,\gamma_1],\dots,[I_m,\gamma_m]$ of $G$-extensions of $K$ as defined in \cref{invariants-section} and let $e_i$ be the size of $I_i$. According to \Cref{fine_tame_remark}, the fine Artin conductor of a tamely ramified homomorphism $f\in S_{K,G}$ with ramification type $[I_i,\gamma_i]$ is
\[
\ffine(f,\psi) = \sum_{\substack{d\mid e_i\\d\neq e_i}} \varphi\left(\frac{e_i}{d}\right)\langle\psi,\gamma_i^d\rangle_{I_i}.
\]

Properties \ref{prop_triv} and \ref{prop_sum} imply that all fine conductors can be determined from those for nontrivial irreducible representations of $G$ defined over $K$. From \Cref{general-conj-remark} and \cite[section 12.4]{serre-representations-of-finite-groups}, we see that the number of such representations is equal to the number $m$ of nontrivial ramification types. Let us call the characters of these representations $\psi_1,\dots,\psi_m$ and let
\[
a_{ij} = \sum_{\substack{d\mid e_i\\d\neq e_i}} \varphi\left(\frac{e_i}{d}\right)\langle\psi_j,\gamma_i^d\rangle_{I_i},
\]
the fine Artin conductor of a homomorphism with ramification type $[I_i,\gamma_i]$ with respect to the character $\psi_j$.

Specifying the ramification type of a tamely ramified homomorphism is equivalent to specifying its fine Artin conductors. More precisely:

\begin{theorem}\label{linear-independence}
The matrix $A=(a_{ij})_{i,j}\in\GL_m(\Q)$ is invertible.
\end{theorem}

\begin{proof}[Proof of \Cref{linear-independence}]
Let $(I,\gamma)$ be a pair as in \Cref{ramification_types_def} and let $e$ be the size of $I$. For any $d\mid e$, recall that $\gamma^d$ is a character of $I$ with image $\mu_{e/d}$. Let $\widetilde{\gamma^d}:G\ra\ol K$ be the average of the $\Gal(K(\mu_{e/d})|K)$-conjugates of the character $\gamma^d$ of $I$. If $s_d$ is the size of $\Gal(K(\mu_{e/d})|K)$, note that the sum $s_d\cdot\widetilde{\gamma^d}$ of the conjugates of $\gamma^d$ is the character of the direct sum of the conjugates of the representation $\gamma^d$. This is the irreducible representation of $I$ defined over $K$ that contains the representation $\gamma^d$ defined over $\ol K$. Every irreducible representation of $I$ defined over $K$ is of the form $s_1\cdot\widetilde{\gamma^d}$ for some isomorphism $\gamma:I\ra\mu_e$ and some $d\mid e$. If the representation of $I$ is faithful, then $d=1$.

Since $\psi_j$ is defined over $K$, we have
\[
a_{ij} = \sum_{\substack{d\mid e_i\\d\neq e_i}} \varphi\left(\frac{e_i}{d}\right)\langle\psi_j,\widetilde{\gamma_i^d}\rangle_{I_i}.
\]
By Frobenius reciprocity,
\[
a_{ij} = \sum_{\substack{d\mid e_i\\d\neq e_i}} \varphi\left(\frac{e_i}{d}\right)\langle\psi_j,\Ind_{I_i}^G\widetilde{\gamma_i^d}\rangle_{G}.
\]
Let $V$ be the $\Q$-vector space of class functions of $G$ with values in $K$ spanned by the $m$ characters
\[
\chi_i = \sum_{\substack{d\mid e_i\\d\neq e_i}}\varphi\left(\frac{e_i}{d}\right)\Ind_{I_i}^G\widetilde{\gamma_i^d}.
\]
Note that the class function $\chi_i$ is independent of the choice of representative $(I_i,\gamma_i)$ of the ramification type $[I_i,\gamma_i]$.

By induction over the size of $I$, we now show that $V$ contains all characters $\Ind_I^G\omega$ of $G$ induced by irreducible faithful characters $\omega$ of nontrivial cyclic subgroups~$I$ of $G$ defined over $K$: As explained above, we can write $\omega$ as a rational multiple of $\widetilde{\gamma}$ for some isomorphism $\gamma:I\ra\mu_e$. We can assume without loss of generality that $I_i=I$ and $\gamma_i=\gamma$. Then, $\chi_i$ is the linear combination of a nonzero rational multiple of $\Ind_I^G\omega$ (from the summand with $d=1$) and characters induced from proper subgroups of $I$ (from the summands with $d>1$). By induction, the latter are contained in $V$. As $\chi_i$ is contained in $V$, it follows that $\Ind_I^G\omega$ indeed lies in $V$.

As every non-faithful character of a cyclic group is induced from a proper subgroup, it actually follows that $V$ contains all characters $\Ind_I^G\omega$ of $G$ induced by nontrivial irreducible characters $\omega$ of cyclic subgroups $I$ of $G$ defined over $K$.

According to \cite[section 12.5]{serre-representations-of-finite-groups}, the $\Q$-vector space spanned by the characters of all representations of $G$ defined over $K$ is also spanned by the characters induced by irreducible characters of cyclic subgroups defined over $K$. The space of spanned by characters of nontrivial irreducible representations is then spanned by characters induced by nontrivial irreducible characters of cyclic subgroups.

To summarize, $V$ is the space spanned by the characters $\psi_1,\dots,\psi_m$ of the nontrivial irreducible representations of $G$ defined over $K$:
\[
\langle\chi_1,\dots,\chi_m\rangle_\Q = \langle\psi_1,\dots,\psi_m\rangle_\Q.
\]
The matrix $A=(a_{ij})_{i,j}=(\langle\chi_i,\psi_j\rangle_G)_{i,j}$ is therefore invertible as the inner product $\langle-,-\rangle_G$ is nondegenerate.
\end{proof}

\begin{remark}
Let $G$ have exponent $n$. Assume that the residue field characteristic of $K$ does not divide $\#G$ and $K$ contains all $n$-th roots of unity. In \Cref{heuristic-local-lemma}, we showed that all $m$ ramification types actually occur, so it follows in particular that the $m$ maps $\ffine(-,\psi_j):S_{K,G}\ra\Z_{\geq0}$ are linearly independent. As explained in \Cref{wood-yasuda-linearly-dependent}, this is in contrast with the weights defined in \cite{wood-yasuda-1}.
\end{remark}

\printbibliography

\end{document}